\documentclass[12pt]{amsart} \usepackage{amscd,amssymb}

\newtheorem{theorem}{Theorem}
\newtheorem{lemma}[theorem]{Lemma}
\newtheorem{corollary}{Corollary}
\newtheorem{proposition}[theorem]{Proposition}

\theoremstyle{definition}                 
            
\newtheorem{remark}{Remark}

\usepackage{amsfonts}
\newcommand{\field}[1]{\mathbb{#1}}          \newcommand{\Q}{\field{Q}}
\newcommand{\R}{\field{R}}                   \newcommand{\Z}{\field{Z}}
\newcommand{\C}{\field{C}}

\newcommand{\D}{\Delta}                    
\newcommand{\g}{\gamma}                         \newcommand{\G}{\Gamma}

 \newcommand{\ra}{\rightarrow}

\begin{document}

\title[The Margulis-Zimmer Conjecture for non-uniform arithmetic groups]
{The Margulis-Zimmer Conjecture for non-uniform arithmetic groups}

\date{}

\author{ Yehuda Shalom,  T.N.Venkataramana, George Willis}

\maketitle{}

The goal of this article  is to investigate various aspects of the Margulis-Zimmer Conjecture on subgroups commensurated by higher rank lattices. In the case of non-uniform arithmetic subgroups of semi-simple Lie groups , the M-Z conjecture has already been proved in a sizeable number of cases. In this article, we first prove the M-Z conjecture for {\it all} higher rank non-uniform lattices. 

\begin{theorem} \label{maintheorem} Every infinite subgroup $N$ of a {\it non-cocompact} irreducible lattice $\G$ in a real semi-simple Lie group $G$ of real rank at least two,  which is commensurated by $\G$,  has finite index in $\G$.
\end{theorem}

In view of the Margulis arithmeticity theorem, this is equivalent to the following. 

\begin{theorem} \label{mainth} Let $G\subset SL_n$ be a linear semi-simple algebraic group defined over $\Q$ which is $\Q$-simple. Assume that $\R-rank (G)\geq 2$ and $\Q-rank (G)\geq 1$. Let $N\subset G(\Z)$ be an infinite subgroup whose commensurator in $G(\Z)$ has finite index in $G(\Z)$. Then $N\subset G(\Z)$ has finite index. 
\end{theorem}

\begin{remark} If $\Q-rank (G)\geq 2$, then this is proved in \cite{Ven 1} using unipotent elements. For $G=R_{K/\Q}({\mathbb  G })$ ($R$ is the Weil restriction of scalars) where $\mathbb G$ is a Chevalley group over a number field $K$  with $K-rank ({\mathbb G})\geq 2$ it is proved in \cite{Sha-Wil}  by using the notion of flat subgroups of commensurators (this latter proof can be extended to prove the above theorem). Therefore, the new part of the above theorem is when $\Q-rank (G)=1$ (even in this case, [Sha-Wil] have proved Theorem \ref{mainth}  when $\mathbb G=SL_2$ over a number field $K$ such that the ring of integers in $K$ has infinitely many units). \\

Let us give some new examples of higher rank lattices which are {\it not} covered by the cited works \cite{Ven 1}  and \cite{Sha-Wil} : $G=SO(n,1)(\Z[{\sqrt 2}])$ for $n\geq 4$, or $G=SL_2(O_D)$ where $O_D$ is an order in a {\it cubic division algebra } $D$ with centre $\Q$, or $G=SU(n,1)(\Z[{\sqrt 3}])$ for $n\geq 2$. These groups have $\Q-rank$ one but higher real rank, and they are not isogenous to $SL_2$ over the ring of integers of a number field with infinitely many units. Theorem \ref{maintheorem} applies to these lattices. \\

The proof here is uniform and does not use the specific nature of the discrete group. The proof is an adaptation of the proof in \cite{Sha-Wil} but instead of using bounded generation for the whole lattice, we use the much weaker bounded generation for integral unipotent radicals of arithmetic groups. \\
\end{remark} 

Every $\Q$-simple group $G$ is a Weil restriction of scalars of the form $G=R_{K/\Q}({\mathcal G})$ where $\mathcal G$ is an absolutely simple algebraic group defined over a number field $K$. We can thus replace $G$ by an absolutely simple algebraic group $G$ -we denote it again by $G$ - defined over a number field $K$. The assumption on the ranks on $G$ translates to $K-rank (G)\geq 1$ and $\infty-rank (G): = \sum _{v\in \infty} K_v-rank (G) \geq 2$ where $\infty$ denotes the set of inequivalent {\it archimedean places} of $G$. We restate Theorem \ref{mainth} as follows.

\begin{theorem} \label{reformulate} Let $G$ be an absolutely simple linear algebraic group defined over a number field $K$ with 
\[K-rank (G)\geq 1, \quad \infty-rank (G)=\sum K_v-rank (G)\geq 2.\] Let $N\subset G(O_K)$ be an infinite subgroup of $G(O_K)$ which is commensurated by $G(O_K)$. Then $N$ has finite index in $G(O_K)$. 
\end{theorem}

Theorem \ref{reformulate} can be generalised to $S$-arithmetic groups as well. Let $S\supset \infty$ be a finite set of places of $K$ and $O_S$ the subring of the field $K$ which remain integral at all places of $K$ outside $S$. 

\begin{theorem} \label{S arithmetic} Suppose $G$ is an absolutely simple linear algebraic group defined over a number field $K$ with $K-rank (G)\geq 1$ and $\infty-rank (G)\geq 2$.  Suppose $N\subset 
G(O_S)$ is an infinite  commensurated subgroup. Then the intersection $N\cap G(O_K)$ has finite index in $G(O_K)$. 
\end{theorem}

Note that in Theorem \ref{S arithmetic} we  assume that the group $G(O_K)$ of integer points is itself a higher rank lattice (i.e. that $\infty-rank(G)\geq 2$); we are unable to handle the case when $S-rank (G)\geq 2$ and $\infty -rank (G)= 1$. For example, we do not know if the analogue of Theorem \ref{S arithmetic} holds for  the group $SL_2(\Z[1/p])$.  \\ 

Theorem \ref{maintheorem} has implications to rigidity of homomorphisms of higher rank lattices into arbitrary locally compact groups. We cite an example here.

\begin{theorem} \label{rigidity} Fix an integer $n\geq 3$ and let $\rho : \G \ra H$ be a homomorphism from a finite index subgroup $\G$ of $SL_n(\Z)$ into a locally compact group $H$. Then either the image of $\rho$ is discrete or else the closure of the image of $\rho$ is a compact profinite group, and is in fact a quotient of the congruence completion of $\G$. 
\end{theorem}

We prove a more general version for general higher rank $S$-arithmetic lattices in Section \ref{homomorphisms}. \\

In a different direction, we prove in section 5, that if $K$ is a global field of positive characteristic, and $G$ is a Chevalley group over $K$ of $K-rank \geq 2$, $v$ a place of $K$, $O$ the subring of $K$ which are integral over all places of $K$ not in equivalent to $v$,  and $\G$ is commensurate to $G(O)$, then 

\begin{theorem} \label{function field} Every infinite commensurated subgroup $N$ of $\G$ has finite index in $\G$. 
\end{theorem}

The method of proof here is via Steinberg relations for elementary matrices.

\newpage
 \section{Generalities on Algebraic Groups }

\subsection{Unipotent Groups}

Suppose $U \subset SL_n$ is a unipotent algebraic group defined over a number field  $K$. The integral unipotent group $U(O_K)$ may be viewed as being commensurate to $U_0(\Z)$ where $U_0=R_{K/Q}(U)$ is the group obtained from $U$ by restriction of scalars. The group $U_0(\Z)$ has bounded generation: that is, there exist elements $u_1,u_2, \cdots, u_k\in U(O_K)$ such that 
\[U_0(\Z)\simeq  U(O_K)= u_1^\Z u_2 ^\Z \cdots u_k ^\Z.\]

Moreover, by \cite{Ragh 2}, any Zariski dense subgroup $\D'$ of $U_0(\Z)$ has finite index in $U_0(\Z)$: part (2) of Theorem (2.1) of \cite{Ragh 2} says that  $U_0(\R)/\D'$ is compact; since $U_0(\Z)/\D'$ is a discrete subset of $U_0(\R)/\D'$, it is finite and hence $\D'\subset U_0(\Z)$ has finite index.  \\

\subsection{Generalities on Groups with $K-rank (G)\geq 1$.} 

Let $P$ be a  parabolic $K$-subgroup with unipotent radical $U$ and fix a Levi decomposition $P=LU$ with $L$ defined over $K$. We will need the following result from \cite{Ragh1} (for groups $G$ with $K-rank (G) \geq 2$) and \cite{Ven 2} (for groups $G$ with $K-rank (G)=1$).

\begin{lemma} \label{unipotent generators} Suppose $G$ is absolutely almost simple with $K-rank (G)\geq 1$ and $\infty-rank (G)\geq 2$. Suppose $P$ is a  parabolic subgroup with unipotent radical $U$. If $\G \subset G(O_K)$ is a Zariski dense subgroup which intersects $U(O_K)$ in a subgroup of finite index, then  $\G$ has finite index in $G(O_K)$. 
\end{lemma}

In the above result, the ``higher rank" assumption ($\infty-rank(G)\geq 2$) is necessary.    

\newpage 

\section{A Result on Integral Unipotent Radicals  of $G(O_K)$}

\subsection{Commensurated Subgroups and a Topology on $G(O_K)$} 

Suppose then that $N\subset \G= G(O_K)$ is an infinite subgroup which is commensurated by $G(O_K)$. A topology may be defined on $\G$ by designating a fundamental system of neighbourhood of identity in $\G$ to be the system of subgroups $N'$ of $N$ of finite index and, for each $x\in \G$, the sets $xN'$ to be a fundamental system of neighbourhoods of $x$. That this topology makes $\G$ a topological group follows because $N$ is commensurated. It also follows that the completion, $\widehat{\G}$, of $\G$ with respect to this topology is a locally compact totally disconnected group and that the closure of $N$ in $\widehat{\G}$ is open and is the profinite completion of $N$. Furthermore, for any subgroup $N'$ of finite index in $N$, the closure $\widehat{N'}$ is open and compact in $\widehat{\G}$ and satisfies that $\G \cap \widehat{N'}=N'$. Details of this and other related completions may be seen in \cite[\S5]{ReidWesolek}, where it is called the \emph{Belyaev completion}.\\

We consider the inner conjugation action of $\G=G(O_K)$ on the completion $\widehat{\G}$.  \\

\subsection{The Unipotent Group $U(O_K)$.}  Let $u\in U(O_K)$ be a unipotent  element of infinite order.
Since $\infty -rank (G)\geq 2$, it follows from Theorem A  of \cite{Lu-Mo-Ra} (see the paragraph before the statement of Theorem B in \cite{Lu-Mo-Ra}) that for a fixed finite set of generators of $G(O_K)$, the word length of the element $u^n$ is $O(\log n)$ (as $n$ tends to infinity). \\

\begin{lemma} \label{unipotent} Let $\infty-rank (G)\geq 2$ and $u\in U(O_K)$ a unipotent element. Then the element $u$ normalises 
a compact open subgroup $V$ in $\widehat{\G}$. \\

Equivalently, the element $u$ normalises a subgroup of $G(O_K)$ which is commensurate to $N$.  
\end{lemma}

This follows from Proposition (6.10) of \cite{Sha-Wil} and the result  (that the word length of $u^n$ is $O(\log n)$)  of \cite{Lu-Mo-Ra} mentioned earlier.\\

We will also need the following result. Fix a finite set $u_1, \cdots, u _k \in U(O_K)$ such that the product 
$\D^+= u_1^\Z \cdots u_k^{\Z}$ contains a subgroup of  finite index in the group $U(O_K)$. The existence of such a finite set  follows from  the fact (mentioned before) that the integral unipotent group $U(O_K)$ has  bounded generation. 

\begin{lemma} \label{unipotent product} There exists a subgroup $\D_1 \subset \D^+$ of finite index in $U(O_K)$ which normalises a compact open subgroup $V$ in $\widehat{\G}$. \\

Equivalently, $\D_1$ normalises  a subgroup $N'$ of $G(O_K)$ which is commensurate to $N$. 

\end{lemma}

\newpage

\section{Proof of Theorem \ref{reformulate} and of Theorem \ref{S arithmetic}}

\subsection{Proof of Theorem \ref{reformulate}.}

Let $N \subset \G=G(O_K)$ be an infinite subgroup commensurated by $\G$. The Zariski closure of $N$ is also infinite and hence the identity component $H$ of the Zariski closure has positive dimension. Moreover, if $N'$ has finite index in $N$ (or is commensurate to $N$), then $H$ is still the identity component of the Zariski closure of $N'$. Consequently, $H$ is normalised by $\G$ and by the Borel density theorem, $G$ normalises $H$. The absolute simplicity of $G$ then implies that $H=G$ and hence $N$ is Zariski dense in $G$. Consequently, any subgroup $N'\subset \G$ commensurate to $N$  is also Zariski dense in $G$. \\

By Lemma \ref{unipotent product} there exists a subgroup $\D_1\subset U(O_K)$ of finite index which normalises a subgroup $N'\subset \G$ such that $N'$ is commensurate to $N$. Consider the group $\D$ generated by $\D_1$ and $N'$. By the preceding paragraph, $\D$ is Zariski dense in $G$. Moreover, the intersection  $\D \cap U(O_K)$ contains $\D_1$ and hence has finite index in $U(O_K)$. By Lemma \ref{unipotent generators}, $\D$ has finite index in $\G$, and {\it normalises} the infinite subgroup $N'$. By the Margulis normal subgroup theorem, $N'$ has finite index in the higher rank arithmetic group  $\D$ and hence in $\G$. \\

The group $N$ being commensurate to $N'$ also has finite index in $\G$. This completes the proof. 

\subsection{Proof of Theorem \ref{S arithmetic}}

Let $N\subset G(O_S)$ be a commensurated infinite subgroup. As before, define the corresponding Belyaev topology on the group $G(O_S)$. We then get a completion $\widehat{G(O_S)}$ such that the closure of $N$ is open and compact. By Lemma \ref{unipotent product}, there exists a subgroup $\D_1 $  of finite index in the {\it integral unipotent group}  $U(O_K)$ ({\bf not} the $S$-integral unipotent group $U(O_S)$) which fixes an open compact subgroup $\Omega$ of $\widehat{G(O_S)}$. Equivalently, the intersection $N'=\Omega \cap G(O_S)$ is commensurate
to $N$ and is normalised by $\D_1$. \\

Fix an infinite order element $n$ in $N'$; for $u\in \D_1$, the conjugate $nun^{-1}$ lies in $G(O_S)$ and is unipotent; hence some integral power of it is in $G(O_K)$. Hence we may assume, $nun^{-1}\in G(O_K)$. Therefore, the commutator $c=nun^{-1}u^{-1}$ lies in $G(O_K)$ and also lies in $N'$ since $u$ normalises $N'$.  We may assume that $c$ has infinite order. \\
 
 Now both $G(O_K)$ and $N'$ are commensurated by $G(O_S)$. Hence their intersection $N_1= N'\cap G(O_K)$ is also commensurated by $G(O_S)$, and in particular by the higher rank group $G(O_K)$, and is infinite by the preceding paragraph. By Theorem \ref{reformulate}, $N_1$ has finite index in $G(O_K)$. Thus $N$ intersects $G(O_K)$ in a subgroup of finite index. This proves Theorem \ref{S arithmetic}. 

\begin{remark} Theorem \ref{S arithmetic} can be restated as  follows. Suppose $S\supset \infty$ is a finite set of places of $K$ and $N\subset G(O_S)$ a commensurated subgroup. Assume that $\infty -rank (G)\geq 2$ and $K-rank (G)\geq 1$. Then $N$ has the standard description of commensurated subgroups. \\

This follows from Lemma (2.8) of \cite{Lub-Zim}: by Theorem \ref{S arithmetic}, we know that the intersection $N\cap G(O_K)$ has finite index in $G(O_K)$. By Lemma (2.8) of \cite{Lub-Zim}, it follows that $N$ is commensurate to a subgroup of the form $G(O_{S'})$  where $S'$ is a subset of the set $S$ of places of $K$, and $S'$  includes all the infinite places. This is the meaning of the phrase "standard description of commensurated subgroups". 

\end{remark}

\begin{remark} The above proof of Theorem \ref{S arithmetic} does not seem to work for $SL_2(\Z[1/p])$. However, what {\it can} be asserted is that an integral unipotent element normalises a commensurated subgroup of $SL_2(\Z[1/p])$. This follows because any such element $x$ is a $p^n$th power for all $n>0$ and so, by Proposition 4 in \cite{Structure94}, $x$ normalises a compact open subgroup in any homomorphic image of $SL_2(\Z[1/p])$.

\end{remark}

\newpage

\section{Homomorphisms of Non-uniform higher rank lattices into Locally Compact Groups} \label{homomorphisms}

\subsection{Generalities on Locally Compact Groups}

Let $H$ be a locally compact group and $H^0$ the connected component of identity in $H$. Then $H^0$ is invariant under all continuous automorphisms of $H$ (in this sense, $H^0$ is a characteristic subgroup of $H$). Let $p: H\ra H/H^0$ be the quotient map. Then $H/H^0$ is a totally disconnected group. \\

By the Gleason-Yamabe Theorem, $H$ contains an open subgroup $U$ with a compact normal subgroup $C$ such that the quotient $L=U/C$ is a Lie group. By replacing $U$ by the inverse image of the open subgroup $L^0$ of $L$ if necessary, we may assume that $L$ is connected. Let $q: U\ra U/C$ be the quotient map. The following Proposition is roughly the same as the  result in \cite{Burger}. 

\begin{proposition} With the preceding notation, we have the following. \\

(1) The group $U$ contains $H^0$.\\

(2) The quotient $U/H^0$ is compact. \\

(3) The group $U$ is commensurated by $H$. \\

(4) We have the equalities  $U=CH^0=H^0C$ and $L=U/C=H^0/H^0\cap C$. 

\end{proposition}

\begin{proof} The group $U\cap H^0$ is an open subgroup of the connected group $H^0$ and therefore is all of $H^0$. Hence $H^0=U\cap H^0\subset U$. \\

Since the quotient map $p: H\ra H/H^0$ is an open map, and $U\supset H^0$, it follows that $p(U)=U/H^0$ is open in $H/H^0$. Since $H/H^0$ is totally disconnected, it follows that so is $U/H^0$. Therefore, by a theorem of Dantzig, the group $U/H^0$ contains a compact open subgroup $W$ of $H/H^0$. Let $V=p^{-1}(W)$. Then $V$ is an open subgroup of $U$ (and hence of $H$) which contains $H^0$. In particular, the image of $V$ under the open map $q: U \ra L$ is an open subgroup; the connectedness of $L$ then implies that $q(V)=q(U)=L$ and hence $U=CV$. Taking the image of this equality under the map  $p$, we see that $p(U)=p(C) p(V)=p(C)W$ and the latter is compact since both $C$ and $W$ are compact.  This proves (2). \\

Since $U/H^0$ is an open compact subgroup of $H/H^0$, it is commensurated by $H/H^0$; since $U\supset H^0$,  this is equivalent to saying $U$ is commensurated by $H$. This proves (3). \\

In the proof of (2), we proved that if $W$ is an open subgroup of $U/H^0$ (the latter is proved to be compact), then $U=Cp^{-1}(W)$. Let $\{W_n\}$ be a sequence of open subgroups of $U/H^0$ shrinking to $\{1\}$ and let $V_n=p^{-1}(W_n)$. Since $U=CV_n$, given $u\in U$, there exists a sequence $c_n\in C$ and a sequence $v_n\in V_n$ such that $u=c_nv_n$. Since $C$ is compact, we may assume that the sequence $\{c_n\}$ converges to $c\in C$, say. Therefore, $v_n$ converges to $v\in U$. Hence $u=cv$. Taking the $p$-images of the equation $u=c_nv_n$, we see that $p(u)=p(c_n)p(v_n)$ and $p(v_n)\in W_n$ tends to the identity. Hence $p(u)=p(c)$, and hence $U=CH^0$. This proves (4). 
\end{proof} 

\subsection{General results on Locally Compact Groups} Let $Q$ be a centre-less connected real semi-simple Lie group without compact factors. Then $Q$ is a direct product of non-compact connected simple Lie groups which have no normal subgroups. If $L$ is a connected Lie group, $R$ is its radical (the maximal connected normal solvable subgroup in $L$) and $Z$ is the centre of the semisimple group $L/R$, then $L/Z$ is a direct product of a centre-less semi-simple group $Q$ without compact factors and a {\it compact}  centre-less semi-simple group $K$. The inverse image of $K$ under the quotient map $L\ra L/Z$ is clearly a maximal (closed) amenable normal   subgroup $N'$ of $L$. \\

We now return to the situation of the preceding subsection: $H$ is a locally compact group, and $U$ is an open subgroup with a surjective map $U\ra L$ to a connected Lie group $L$, with compact kernel $C$. By the proposition in the preceding subsection, we have $U\supset H^0$ and a surjection $q: H^0 \ra L$ with kernel $C\cap H^0$. Hence $N=q^{-1}(N')$ is a maximal amenable normal closed subgroup of $H^0$. Since such a group $N$ is unique, it follows that $N$ is a characteristic subgroup of $H^0$, and is invariant under conjugation by elements of $H$. The following Theorem is also very similar to a result in \cite{Burger}.

\begin{theorem} \label{semisimplequotient} Let $H$ be a locally compact group and $H^0$ its identity component, and $N$ the maximal closed normal {\bf amenable} subgroup of the identity component $H^0$. Then $H$ contains an open subgroup $H'$ of finite index such that $H'/N$ is a direct product of the form 
\[H'/N=Q\times H'/H^0,\]
where $Q$ is a connected centre-less semi-simple Lie group without compact factors (and $H'/H^0$ is totally disconnected). 

\end{theorem}

\begin{remark} It can happen that $Q$ is trivial. This is equivalent to saying that $H^0=N$ is  amenable. 
\end{remark}

\begin{proof} We first note that the group $Aut (Q)/Int(Q)$ of outer automorphisms of the semi-simple group $Q$ is finite. Since $Q=H^0/N$ is a characteristic quotient of $H^0$, it follows that the conjugation action of $H$ on $H^0$ descends to an action of $H$ on $Q=H^0/N$. Hence there exists a subgroup $H'$ of finite index in $H$ which acts on $Q$ by inner automorphisms. We will henceforth assume, as we may, that $H'=H$. \\

Given $h\in H/N$, then there exists $q\in H^0/N=Q$ such that $hxh^{-1}=qxq^{-1}$ for all $x\in Q$. Therefore, $hq^{-1} \in Z(Q)$ where $Z=Z(Q)$ is the centraliser of the normal subgroup $Q$ in $H/N$. Since $Q$ is centre-less, it follows that $H/N=ZQ$ is a direct product:

\[H/N=Z\times H^0/N.\] 
Therefore,  $Z=(H/N)/Q=(H/N)/(H^0/N)=H/H^0$. This proves the theorem. 
\end{proof}

Theorem \ref{semisimplequotient} has an implication for higher rank lattices in groups over non-archimedean local fields (of arbitrary characteristic). 

\begin{corollary} Let $\G$ be an irreducible lattice in a product $G=\prod _{i\in I} G_i(k_i)$, where $I$ is a finite set and for each $i\in I$, $k_i$ is a non-archimedean local field, with $rank (G):= \sum_{i\in I} k_i-rank (G_i)\geq 2$. Let $\rho : \G \ra H$ be a homomorphism of $\G$ into a locally compact group $H$ with non-discrete and dense image. Then the connected component $H^0$ of identity in $H$ is an amenable group. \\

For example, if and $n\geq 3$, ${\mathbb F}_p$ is the finite field with $p$-elements, and $\rho : SL_n({\mathbb F}_p[t]) \ra H$ is a dense non-discrete homomorphism into a locally compact group $H$, then $H^0$ is amenable.

\end{corollary} 

\begin{proof} Let $N$ be the maximal closed amenable subgroup of $H^0$. By Theorem \ref{semisimplequotient} $H/ N$ is a product $H/N=Q\times H/H^0$ where $H$ is a real semisimple LIe group without compact factors. The projection of $\rho (\G)$ to $Q$ gives a dense homomorphism from $\G$ into $Q$; by the Margulis superrigidity, $Q$ is then trivial and therefore, $H^0=N$ is amenable and $H/N$ is totally disconnected. 
\end{proof}

\begin{remark} Many of the results in the last two subsections are due to M.Burger ( see  \cite{Burger} ).

\end{remark}

\subsection{Lattices in Simple Lie Groups} 
In this section, we make a strong assumption on $\G$: that it is a lattice in a higher rank {\bf simple} Lie group.\\

Suppose $\G$ is an irreducible {\bf non-uniform} lattice in a higher real rank {\it simple} Lie group $G$. By general theory, there exists a number field $K$ and an absolutely almost simple linear algebraic group $\mathbb G$ defined over $K$ such that $K-rank ({\mathbb  G})\geq 1$  (and $\infty-rank ({\mathbb  G})\geq 2$)  
and a surjection $f$  from the group ${\mathbb  G}(K\otimes _\Q \R)$  onto $G$ with finite kernel, such that $f({\mathbb G}(O_K))$ is commensurable to $\G$. \\

This means that $K$ has only one archimedean completion $K_v$  i.e. $K=\Q$ and $K_v=\R$ or else  $K$ is an imaginary quadratic extension of $\Q$ and $K_v=\C$. \\

\begin{lemma} Suppose $\rho: \G \ra C$ is a homomorphism from $\G$ into a connected simple  Lie group $C$ with dense image. Then $C$ is trivial.
\end{lemma}

\begin{proof} The group $C$ is the group of $k$-points of an absolutely simple algebraic group $H$ defined over an archimedean field $k$ ($k$ is either $\R$ or $\C$). The density of the image of $\rho$ means that the image of $\rho$ is Zariski dense in $H$. \\

By Theorem (C), Chapter VIII of Margulis' book, there exists a homomorphism $K\ra k$ of fields (so that $k=K_v$ or $\R$) and a homomorphism $\eta$ of algebraic groups $R_{K_v/k} ({\mathbb G}) \ra H$ defined over $k$ , and a homomorphism $\nu: \G \ra Z(H)$ as before such that $\rho (\g)=\nu (\g)\eta (\g)$. Note that $Z(H)$ is finite. \\

By passing to a subgroup of finite index, we may assume that $\rho (\g)=\eta (\g)$ for all $\g \in \G$. But $\eta (\G)$ is commensurable to $G(O_K)$ which is a discrete subgroup of $R_{K_v/k}(\mathbb G)$ and hence $\rho (\G)=\eta (\G)$ is a discrete subgroup of $C=H(k)$ if $\eta$ is an isomorphism. This  contradicts the density of $\rho(\G)$ in $C$ in the real topology. Therefore, $H$ is trivial and hence $C$ is trivial. 
\end{proof}

\begin{theorem} Suppose $\G$ is an irreducible non-uniform lattice in a higher rank simple Lie group $G$ and $\rho :\G \ra H$ is a homomorphism of $\G$ into a locally compact group $H$ with {\bf dense} and {\bf indiscrete} image. Then $H$ is compact and profinite.
\end{theorem}

\begin{proof} Let $N$ be the maximal closed normal amenable subgroup of the group $H^0$,  the connected component of identity in $H$. Then by theorem \ref{semisimplequotient}, $Q= H^0/N$ is a product of non-compact simple Lie groups and  $H/N=Q\times H/H^0$ is a direct product. Hence  the image of $\G$ in $H/H^0$ is also dense. Being totally disconnected,  $H/H^0$ contains a compact open subgroup $V$ which is commensurated by $H/H^0$. Let $q: \G \ra H\ra H/H^0$ be the composite map. Then $\D=q^{-1}(V)$ is commensurated by  $\G$, and by the main theorem, $\D$ has finite index in $\G$. $V$ has finite index in $H/H^0$. Therefore, $H/H^0$ is compact. \\

By the proposition proved above, $Q$ is trivial, and hence by Theorem \ref{semisimplequotient} $H^0=N$ is amenable, and since $H/H^0$ is compact by the preceding paragraph, $H$ is also amenable and contains $\rho (\G)$ as a dense subgroup. \\

But $\G$ being a lattice in a higher rank lattice in a simple Lie group of real rank greater than one, has Kazhdan's property (T) and hence so does $\rho (\G)$; this means that $H$ is both amenable and has property (T) and is therefore compact.  \\

Let $\theta : H \ra GL_n(\C)$ be a complex finite dimensional representation. The image $Im (\theta)$ of $\theta $ is also  compact and hence a subgroup of finite index in $\G$ maps to a connected compact group $C$. Since  $C$ is an almost direct product of a torus and a connected semi-simple group, the above lemma shows that the semi-simple part is trivial and hence $C$ is a torus. But since a subgroup of finite index in $\G$ has property (T), it follows that the image cannot be a torus, and hence $C$ is trivial. Therefore, the image of any finite dimensional representation $\theta $ is finite and hence $H$ is profinite. 

\end{proof}

Theorem \ref{rigidity} follows as a corollary since for $n\geq 3$, $SL_n(\Z)$ is a lattice in a higher rank simple group, namely $SL_n(\R)$. Moreover, since for $n\geq 3$, $SL_n(\Z)$ has the congruence subgroup property, it follows that $H$ is a quotient of the congruence completion of $\G$.

\subsection{Higher Rank S-Arithmetic Lattices}

For the sake of completeness, we state the corresponding result for $S$-arithmetic groups as well. Fix a finite set $I$ and for each $i\in I$, let $k_i$ be a local field of characteristic zero, let $G_i$ be an absolutely almost simple connected linear algebraic group defined over $k_i$; then $G_i(k_i)$ is a locally compact group in a natural fashion. Assume $k_i-rank (G_i) \geq 1$ for each $i $ (i.e. $G_i(k_i)$ is non-compact).  Let $\Gamma \subset G:= \prod _{i\in I} G_i(k_i)$ be an irreducible lattice. Define $rank (G)=\sum k_i-rank (G_i)$ and suppose $rank (G)\geq 2$. By the Margulis arithmeticity theorem, there exists a number field $K$, a finite set $S$ of places of $K$ including all the archimedean places  of $K$, and a homomorphism $\phi : \prod _{v\in S} G(K_v) \ra G$ with compact kernel and compact cokernel such that the image $\phi (G(O_S))$ and $\Gamma $ are commensurable. \\

Let $G_i^* \ra G_i$ be the simply connected cover of $G_i$; then the image of $G_i^*(k_i) $ in $G_i(k_i)$ is cocompact and the quotient $Q_i$ is a compact torsion group. Hence,  a finitely generated subgroup of $Q_i$ is finite. Since the lattice $\G$ of the preceding paragraph is finitely generated, it follows that after replacing it by a finite index subgroup if necessary, we  assume, as we may,  that each $G_i$ is simply connected.  We may also assume that $G$ is simply connected. We consider the subset $S_n\subset S$ consisting of those places of $S$ where $G(K_v)$ is non-compact; then the map $\prod _{v\in S_c} G(K_v) \ra \prod _{i\in I} G_i(k_i)$ may be seen to be an isomorphism, induced by an  isomorphism $K_v\ra k_i$ of local fields for some $v\in S_n$ and an isomorphism $G\ra G_i$ of simply connected algebraic groups over $k_i$. \\

We may write $S_n(\infty)$ and $S_n(f)$ for the set of archimedean and finite places of $S_n$. Suppose $T\subset S_n$ is a {\it proper subset}. Then strong approximation then says (recall that $G$ is assumed to be simply connected) that  the projection of $\G$ to the factor $G_T=\prod _{v\in T} G(K_v)$ is dense, and any quotient $G_T/N_T$ of $G_T$ by a {\it central normal subgroup}  also contains $\G$ as a dense subgroup. We will refer to the embeddings of $\G$ into $G_T/N_T$ (as $T$ varies) as the standard dense embeddings. We let $T(f)$ be the set of finite places in $T$ and consider $G_{T(f)}/N_{T(f)}$. \\

We will say that a lattice $\G \simeq G(O_S)$ as above has the standard description of commensurated subgroups if every infinite subgroup $N\subset \G$ which is commensurated by $\G$ is of the form $G(O_S')$ for some subset $S'\subset S$ of places containing all the archimedean completions of $K$. IT is easy to show from this that every locally compact totally disconnected group $H$ containing $\G$ densely is of the form $G_{T(f)}/N_{T(f)}$ for some $T(f)\subset T$ and $N_T$ as in the preceding paragraph. In fact, more is true: 

\begin{theorem} Suppose $\G$ is a higher rank $S$-arithmetic group as above and suppose $\G$ has the standard description of commensurated subgroups. Let $\rho : \G \ra H$ be a homomorphism with dense and non-discrete image. If $N\subset H^0$ is the maximal normal amenable subgroup, then the quotient $H/N$ is of the form $G_T/N_T$ for some $T$ as above. 
\end{theorem}

\newpage 

\section{The proof of Theorem \ref{function field}}

\newpage 
\section{Acknowledgments} T.N.Venkataramana gratefully acknowledges the support of the Raja Ramanna Fellowship and the hospitality and fantastic working conditions at the ICTS, Bangalore, where this work was done.

\end{document}